\documentclass[12pt]{amsart}

\usepackage{amsmath,amssymb,amsfonts,amsthm}
\usepackage{mathtools}
\usepackage[margin=1in]{geometry}

\usepackage{graphicx}
\usepackage{tikz}
\usetikzlibrary{arrows.meta,graphs,shapes.geometric,decorations.pathreplacing}
\usepackage{pgfplots}
\pgfplotsset{compat=1.18}

\usepackage{booktabs}
\usepackage{multicol}
\usepackage{enumitem}

\usepackage[numbers,sort&compress]{natbib}

\usepackage{xcolor}
\definecolor{myteal}{RGB}{0,128,128}
\usepackage[pdfborder={0 0 0}]{hyperref}
\hypersetup{
  colorlinks=true,
  citecolor=myteal,
  linkcolor=blue,
  urlcolor=myteal
}

\usepackage{aliascnt}
\usepackage[capitalise, noabbrev]{cleveref}

\usepackage{parskip}

\usepackage{listings}
\usepackage{booktabs}

\newcommand{\newaliastheorem}[3]{%
  \newaliascnt{#1}{thm}%
  \newtheorem{#1}[#1]{#2}%
  \aliascntresetthe{#1}%
  \crefname{#1}{#2}{#3}%
  \Crefname{#1}{#2}{#3}%
}

\theoremstyle{plain}
\newtheorem{thm}{Theorem}[section]
\crefname{thm}{Theorem}{Theorems}
\Crefname{thm}{Theorem}{Theorems}

\newaliastheorem{lem}{Lemma}{Lemmas}
\newaliastheorem{prop}{Proposition}{Propositions}
\newaliastheorem{cor}{Corollary}{Corollaries}
\newaliastheorem{clm}{Claim}{Claims}
\newaliastheorem{conj}{Conjecture}{Conjectures}

\theoremstyle{definition}
\newaliastheorem{defn}{Definition}{Definitions}
\newaliastheorem{ex}{Example}{Examples}
\newaliastheorem{notation}{Notation}{Notations}
\newaliastheorem{setup}{Set-up}{Set-ups}
\newaliastheorem{remark}{Remark}{Remarks}
\newaliastheorem{problem}{Problem}{Problems}
\newaliastheorem{question}{Question}{Questions}
\newaliastheorem{observation}{Observation}{Observation}

\setlist[enumerate]{leftmargin=*,itemsep=2pt,topsep=4pt}
\SetEnumerateShortLabel{a}{\textup{(\alph*)}}
\SetEnumerateShortLabel{A}{\textup{(\Alph*)}}
\SetEnumerateShortLabel{1}{\textup{(\arabic*)}}
\SetEnumerateShortLabel{i}{\textup{(\roman*)}}
\SetEnumerateShortLabel{I}{\textup{(\Roman*)}}

\makeatletter
\newtheorem*{rep@theorem}{\rep@title}
\newcommand{\newreptheorem}[2]{%
\newenvironment{rep#1}[1]{%
 \def\rep@title{#2 \ref{##1}}%
 \begin{rep@theorem}}%
 {\end{rep@theorem}}}
\makeatother

\newreptheorem{theorem}{Theorem}

\DeclareMathOperator{\ord}{ord}

\begin{document}

\title{Pseudo-Gorenstein$^{*}$ Graphs}

\author[Hibi]{Takayuki Hibi}
\address[T.~Hibi]{Department of Pure and Applied Mathematics, Graduate School of Information Science and Technology, Osaka University, Suita, Osaka 565--0871, Japan}
\email{\href{mailto:hibi@math.sci.osaka-u.ac.jp}{hibi@math.sci.osaka-u.ac.jp}}

\author[Kara]{Selvi Kara}
\address[S.~Kara]{Department of Mathematics, Bryn Mawr College, Bryn Mawr, PA 19010}
\email{\href{mailto:skara@brynmawr.edu}{skara@brynmawr.edu}}

\author[Vien]{Dalena Vien}
\address[D.~Vien]{Department of Mathematics, Bryn Mawr College, Bryn Mawr, PA 19010}
\email{\href{mailto:dvien@brynmawr.edu}{dvien@brynmawr.edu}}

	

    	   	
	\thanks{The present paper was completed while the first author stayed at Bryn Mawr College, Pennsylvania, February 28 -- March 21, 2026. The second author Kara was supported by NSF grant DMS--2418805. } 

\begin{abstract}
Motivated by pseudo-Gorenstein rings in commutative algebra, introduced by Herzog et al., we define pseudo-Gorenstein$^{*}$ graphs and classify them in several natural graph families using independence polynomials.


\end{abstract}

\maketitle

\section*{Introduction}

Let \(A=\bigoplus_{n\ge 0}A_n\) be a standard graded \(K\)-algebra. Its Hilbert series is
\[
H_A(t)=\sum_{n\ge 0}(\dim_K A_n)t^n,
\]
and it can be written uniquely in the form
\[
H_A(t)=\frac{h_0+h_1t+\cdots+h_st^s}{(1-t)^d},
\]
where \(d=\dim A\) and \(h_s\neq 0\). The numerator
\[
h_A(t)=h_0+h_1t+\cdots+h_st^s
\]
is called the \emph{\(h\)-polynomial} of \(A\). The \(\mathfrak a\)-invariant of \(A\) is introduced by Goto and Watanabe  in \cite{GW} and it is defined as
\[
\mathfrak a(A)=\deg h_A(t)-d=s-d.
\]
Hilbert functions and \(h\)-vectors play a central role in commutative algebra and combinatorics; see \cite{Sta78}. If \(A\) is Cohen--Macaulay, then the coefficients of \(h_A(t)\) are positive, while if \(A\) is Gorenstein, then the \(h\)-vector is symmetric, and in particular \(h_s=1\). Motivated by the pseudo-Gorenstein condition studied in \cite{EHHM}, and by earlier combinatorial work of Stanley on the top coefficient of the \(h\)-polynomial \cite{Sta72,Hibi,EC1}, we investigate a graph-theoretic analogue of this extremal behavior.

Throughout the paper, graphs are finite and simple. Let \(G = (V(G),E(G))\) be a graph on  \(V(G) =[n]\), let \(S=K[x_1,\dots,x_n]\), and let \(I(G)\subset S\) be the edge ideal of \(G\), generated by the monomials \(x_ix_j\) with \(\{i,j\}\in E(G)\). Write
\[
h_{S/I(G)}(t)=h_0+h_1t+\cdots+h_st^s.
\]
We say that \(G\) is \emph{pseudo-Gorenstein} if \(h_s=1\), and \emph{pseudo-Gorenstein\(^{*}\)} if, in addition, \(\mathfrak a(S/I(G))=0\). Thus the pseudo-Gorenstein\(^{*}\) condition requires that the \(h\)-polynomial have top coefficient \(1\) and maximal possible degree. In contrast with the commutative-algebraic setting, we do not assume that \(S/I(G)\) is Cohen--Macaulay.

The independence polynomial provides the natural framework for our problem. While this perspective has only recently been introduced into the study of $h$-polynomials of edge and cover ideals \cite{biermann2026}, the present paper shows that it is a powerful and effective tool for understanding pseudo-Gorenstein$^{*}$ graphs.  Let \(\alpha(G)\) denote the independence number of \(G\), and let
\[
P_G(x)=\sum_{i=0}^{\alpha(G)} g_i x^i
\]
be the independence polynomial of \(G\), where \(g_i\) is the number of independent sets of \(G\) of size \(i\). Since \(S/I(G)\) is the Stanley--Reisner ring of the independence complex of \(G\), one has
\[
h_{S/I(G)}(t)=(1-t)^{\alpha(G)}P_G\!\left(\frac{t}{1-t}\right).
\]
In particular, the value of $P_G(x)$ at \(x=-1\) controls the top coefficient of the \(h\)-polynomial. Moreover, it was shown in \cite{biermann2026} that if $M(G)$  denotes the multiplicity of $x=-1$ as a root of the independence polynomial $P_G(x)$,
then
\[
\deg h_{S/I(G)}(t)=\alpha(G)-M(G).
\]
Consequently,
\[
\mathfrak a(G)=0 \iff P_G(-1)\neq 0.
\]
Hence
\[
G \text{ is pseudo-Gorenstein}^{*}
\iff
P_G(-1)=(-1)^{\alpha(G)}.
\]
Thus the problem of identifying pseudo-Gorenstein\(^{*}\) graphs becomes a problem about the value  of the independence polynomial at $x=-1$.

The independence polynomial, introduced by Gutman and Harary as a generalization of the matching polynomial, has been studied extensively from several points of view, including real-rootedness, unimodality, coefficient behavior, and special evaluations \cite{GutmanHarary,Claw_free_ind_poly,Unimodal_ind_poly}. The independence polynomial has found applications even outside of graph theory, for instance in music \cite{Application_music_ind_poly}. It is also closely related to other classical graph polynomials: it appears in matching theory through line graphs and also as a specialization of the generalized chromatic polynomial \cite{Gen_choromatic_poly}. Among its special values, \(P_G(1)\) counts all independent sets of \(G\) and is often called the Fibonacci number of \(G\) (see \cite{Fibonaci_main,Fibonacci,Unicyclic}), while \(P_G(-1)\) is the alternating sum of sizes of independent sets (see \cite{levit2009independence_preprint,Levit_Mandrescu,Ind_complex}). For our purposes, this latter value is extremely important.

In this paper, we classify pseudo-Gorenstein\(^{*}\) graphs in several natural families. We prove that a cycle \(C_n\) is pseudo-Gorenstein\(^{*}\) if and only if \(n\equiv 1,2,5,10 \pmod{12}\), 
and that a path \(P_n\) is pseudo-Gorenstein\(^{*}\) if and only if \(n\equiv 0,2,9,11 \pmod{12}\). 
We also show that a complete multipartite graph is pseudo-Gorenstein\(^{*}\) precisely when it is complete bipartite and its largest part has odd size, and we obtain a simple parity criterion for Cameron--Walker graphs. Finally, we study the behavior of the pseudo-Gorenstein\(^{*}\) property under \(C\)-suspension: we prove preservation under suspensions over  vertex covers, determine exactly when full suspensions of cycles and paths are pseudo-Gorenstein\(^{*}\), and classify suspensions over maximal independent sets for cycles and paths. Throughout, the main tools are the deletion--contraction recursion for independence polynomials and the relationship between \(P_G(-1)\), the multiplicity of the root \(-1\), and the degree of the \(h\)-polynomial.

\section{Preliminaries}

We retain the notation of the introduction. For a graph \(G\), write $h_G(t):=h_{S/I(G)}(t)$ and
\[
h_G(t)=\sum_{i=0}^{\alpha(G)} h_i(G)t^i
\]
where \(h_i(G)=0\) for \(i>\deg h_G(t)\). We also use the notation $\mathfrak a(G) := \mathfrak a(S/I(G))$ for the $\mathfrak a$-invariant of $G$ where
\[
\mathfrak a(G)=\deg h_G(t)-\alpha(G),
\]
and
\[
M(G):=\ord_{x=-1}P_G(x),
\]
the multiplicity of \(-1\) as a root of the independence polynomial \(P_G(x)\).

The following identity expresses the coefficient of $t^{\alpha(G)}$ in the \(h\)-polynomial of \(S/I(G)\) in terms of the
independence polynomial of \(G\).

\begin{prop}\label{prop:h-from-independence}
Let \(\alpha=\alpha(G)\). We have
\(h_{\alpha}(G)=(-1)^{\alpha}P_G(-1)\).
\end{prop}

\begin{proof}
Since \(S/I(G)\) is the Stanley--Reisner ring of the independence complex of \(G\),
the standard relation between the \(f\)-vector and \(h\)-vector yields
\[
h_G(t)=(1-t)^{\alpha}P_G\!\left(\frac{t}{1-t}\right)=\sum_{i=0}^{\alpha} g_i t^i(1-t)^{\alpha-i},
\]
where \(g_i\) is the number of independent sets of \(G\) of size \(i\).  Taking the coefficient of \(t^\alpha\)
gives
\[
h_{\alpha}(G)=\sum_{i=0}^{\alpha} g_i(-1)^{\alpha-i}
=(-1)^{\alpha}P_G(-1). \qedhere
\]
\end{proof}
The next theorem, proved in \cite[Theorem 4.4]{biermann2026}, connects the degree of
the \(h\)-polynomial with the root \(-1\) of the independence polynomial.
\begin{thm}\cite[Theorem 4.4]{biermann2026}\label{thm:deg-via-ord}
For every finite simple graph \(G\),
\[
\deg h_G(t)=\alpha(G)-M(G).
\]
\end{thm}

The following is an immediate corollary of the above result from \cite{biermann2026}.

\begin{cor}\label{cor:a-via-M}
For every finite simple graph \(G\), we have \(\mathfrak a(G)=-M(G)\). In particular,
\[
\mathfrak a(G)=0 \iff P_G(-1)\neq 0.
\]
\end{cor}


Combining the definition of pseudo-Gorenstein\(^{*}\) from the introduction with
\Cref{prop:h-from-independence} and \Cref{cor:a-via-M}, we obtain the following
criterion.

\begin{cor}\label{cor:PGstar-criterion}
A finite simple graph \(G\) is pseudo-Gorenstein\(^{*}\) if and only if \(P_G(-1)=(-1)^{\alpha(G)}\).
\end{cor}


Lastly, we recall the following standard vertex-deletion recursion for independence polynomials.

\begin{lem}\label{lem:delete_contract}
Let \(G\) be a finite simple graph and let \(v\in V(G)\). Then
\[
P_G(x)=P_{G-v}(x)+x\,P_{G-N[v]}(x).
\]
\end{lem}

\section{Pseudo-Gorenstein$^{*}$ Cycles and Paths}\label{sec:cycles_paths}

In this section, we focus on paths and cycles and investigate when these fundamental classes of graph families are pseudo-Gorenstein$^{*}$. In view of \Cref{cor:PGstar-criterion}, it is enough to
determine the value of the independence polynomial at \(x=-1\).

The next two lemmas record the values of the independence polynomials of paths
and cycles at \(x=-1\), which will be used repeatedly later. They can be derived
from \cite[Corollaries 4.8 and 5.6]{Kara_Vien_2026}, so we omit the proofs.
 For convenience, let \(P_0\) denote the empty graph, so that  \(P_{P_0}(x)=1\).

\begin{lem}\label{lem:path-minus-one}
Let \(p_n:=P_{P_n}(-1)\) for \(n\ge 0\). Then
\[
p_0=1,\qquad p_1=0,\qquad p_n=p_{n-1}-p_{n-2}\ \ \text{for } n\ge 2.
\]
Consequently, \(p_n\) is periodic of period \(6\), and
\[
p_n=
\begin{cases}
1,& n\equiv 0,5 \pmod 6,\\[4pt]
0,& n\equiv 1,4 \pmod 6,\\[4pt]
-1,& n\equiv 2,3 \pmod 6.
\end{cases}
\]
\end{lem}


\begin{lem}\label{lem:cycle-minus-one}
Let \(c_n:=P_{C_n}(-1)\) for \(n\ge 3\). Then
\[
c_n=p_{n-1}-p_{n-3},
\]
where \(p_n\) is as in \Cref{lem:path-minus-one}. In particular, \(c_n\) is
periodic of period \(6\), and
\[
c_n=
\begin{cases}
2,& n\equiv 0 \pmod 6,\\[4pt]
1,& n\equiv 1,5 \pmod 6,\\[4pt]
-1,& n\equiv 2,4 \pmod 6,\\[4pt]
-2,& n\equiv 3 \pmod 6.
\end{cases}
\]
\end{lem}


Combined with \Cref{cor:PGstar-criterion}, these evaluations reduce the classification of pseudo-Gorenstein\(^{*}\) cycles and paths to a finite congruence check modulo \(12\).

\begin{thm}\label{thm:cycle}
A cycle \(C_n\) is pseudo-Gorenstein$^{*}$ if and only if \(n\equiv 1,2,5,10 \pmod{12}\).
\end{thm}

\begin{proof}
Let \(\alpha=\alpha(C_n)=\lfloor n/2\rfloor\). By \Cref{cor:PGstar-criterion},
\(C_n\) is pseudo-Gorenstein$^{*}$ if and only if
\[
P_{C_n}(-1)=(-1)^{\alpha}.
\]
By \Cref{lem:cycle-minus-one}, \(P_{C_n}(-1)\) is periodic modulo \(6\), while
\((-1)^{\lfloor n/2\rfloor}\) is periodic modulo \(4\). Hence the comparison is
periodic modulo \(12\). Checking \(n=0,1,\dots,11\), we obtain
\[
\begin{array}{c|cccccccccccc}
n \pmod{12} & 0&1&2&3&4&5&6&7&8&9&10&11\\
\hline
P_{C_n}(-1)
&2&1&-1&-2&-1&1&2&1&-1&-2&-1&1\\
\hline
(-1)^{\lfloor n/2\rfloor}
&1&1&-1&-1&1&1&-1&-1&1&1&-1&-1
\end{array}
\]
Thus \(P_{C_n}(-1)=(-1)^{\alpha(C_n)}\) exactly when \(n\equiv 1,2,5,10 \pmod{12}\).
\end{proof}

\begin{thm}\label{thm:path}
A path \(P_n\) is pseudo-Gorenstein$^{*}$ if and only if \(n\equiv 0,2,9,11 \pmod{12}\).
\end{thm}

\begin{proof}
Let \(\alpha=\alpha(P_n)=\lceil n/2\rceil\). By \Cref{cor:PGstar-criterion},
\(P_n\) is pseudo-Gorenstein$^{*}$ if and only if
\[
P_{P_n}(-1)=(-1)^{\alpha}.
\]
By \Cref{lem:path-minus-one}, \(P_{P_n}(-1)\) is periodic modulo \(6\), while
\((-1)^{\lceil n/2\rceil}\) is periodic modulo \(4\). Hence the comparison is
periodic modulo \(12\). Checking \(n=0,1,\dots,11\), we obtain
\[
\begin{array}{c|cccccccccccc}
n \pmod{12} & 0&1&2&3&4&5&6&7&8&9&10&11\\
\hline
P_{P_n}(-1)
&1&0&-1&-1&0&1&1&0&-1&-1&0&1\\
\hline
(-1)^{\lceil n/2\rceil}
&1&-1&-1&1&1&-1&-1&1&1&-1&-1&1
\end{array}
\]
Thus \(P_{P_n}(-1)=(-1)^{\alpha(P_n)}\) exactly when \(n\equiv 0,2,9,11 \pmod{12}\).
\end{proof}

\begin{notation}
For later use, define
\[
a_n:=(-1)^{\lfloor n/2\rfloor}P_{C_n}(-1),
\qquad
b_n:=(-1)^{\lceil n/2\rceil}P_{P_n}(-1).
\]
\end{notation}

\begin{cor}\label{cor:path-cycle-signed-values}
Both sequences  $(a_n)$ and $(b_n)$ are periodic modulo \(12\), with
\[
a_n=
\begin{cases}
1,& n\equiv 1,2,5,10 \pmod{12},\\[4pt]
-1,& n\equiv 4,7,8,11 \pmod{12},\\[4pt]
2,& n\equiv 0,3 \pmod{12},\\[4pt]
-2,& n\equiv 6,9 \pmod{12},
\end{cases} \qquad \text{ and }\qquad 
b_n=
\begin{cases}
1,& n\equiv 0,2,9,11 \pmod{12},\\[4pt]
-1,& n\equiv 3,5,6,8 \pmod{12},\\[4pt]
0,& n\equiv 1,4,7,10 \pmod{12}.
\end{cases}
\]
\end{cor}

We next turn to another natural family of graphs, namely complete multipartite graphs, and determine exactly when they are pseudo-Gorenstein$^{*}$.

\section{Pseudo-Gorenstein$^{*}$ Multipartite Graphs}

Let $G=K_{m_1,m_2,\dots,m_k}$
be the complete multipartite graph with vertex partition
\[
V(G)=V_1\sqcup V_2\sqcup \cdots \sqcup V_k,
\qquad |V_i|=m_i \ge 1,
\]
and edge set
\[
E(G)=\{\{a,b\}: a\in V_i,\ b\in V_j,\ i\neq j\}.
\]
Set \(\alpha:=\alpha(G)\) and  notice that $\alpha (G)=\max\{m_1,\dots,m_k\}$. Then \(\dim S/I(G)=\alpha\).

\begin{lem}\label{lem:multu_indep_poly}
The independence polynomial of \(G\) is
    \[
    P_G(x)
    =
    1+\sum_{i=1}^k \sum_{r=1}^{m_i}\binom{m_i}{r}x^r
    =
    \sum_{i=1}^k (1+x)^{m_i}-(k-1).
    \]    
\end{lem}

\begin{proof}
An independent set in a complete multipartite graph cannot contain vertices
from two different parts, since every vertex in \(V_i\) is adjacent to every
vertex in \(V_j\) whenever \(i\neq j\). Hence every independent set is either
the empty set or a nonempty subset of exactly one part \(V_i\). Therefore, for each \(r\ge 1\), the number of independent sets of size  \(r\) is $\sum_{i=1}^k \binom{m_i}{r}$,
where \(\binom{m_i}{r}=0\) if \(r>m_i\). It follows that
\[
P_G(x)
=
1+\sum_{r\ge 1}\left(\sum_{i=1}^k \binom{m_i}{r}\right)x^r
=
1+\sum_{i=1}^k \sum_{r=1}^{m_i}\binom{m_i}{r}x^r=\sum_{i=1}^k(1+x)^{m_i}-(k-1)\]
where the last equality follows from the binomial theorem.
\end{proof}

\begin{cor}
The complete multipartite graph \(G\) is pseudo-Gorenstein\(^{*}\) if and only if
\(k=2\) and \(\alpha\) is odd.
\end{cor}

\begin{proof}
 We first observe that  $ \mathfrak a (G)=0$ when $k\ge 2$; otherwise, $\mathfrak a (G)=-m_1$. To see this, evaluate $P_G(x)$ from \Cref{lem:multu_indep_poly} at \(x=-1\). Since each \(m_i\ge 1\), we have \(P_G(-1)= 1-k\).  If \(k\ge 2\), then \(P_G(-1)\neq 0\), so \(-1\) is not a root and $M(G)=0$. Thus $\mathfrak a (G)=-M(G)=0$. If \(k=1\), then \(M(G)=m_1\) and $\mathfrak a (G)<0$ in this case.  Therefore,  $\mathfrak a (G)=0$ exactly when $k\geq 2$. 

To complete the proof, it remains to understand when  
\[
h_{\alpha}(G)=(-1)^{\alpha}P_G(-1)=(-1)^{\alpha}(1-k)=1.
\]
This happens when $k=2$ and $\alpha$ is odd.
\end{proof}

\section{Pseudo-Gorenstein$^{*}$ Cameron--Walker Graphs}

In this section we consider connected Cameron--Walker graphs that are neither stars nor star triangles. By the structure theorem of \cite{CW+Graphs}, every such graph is obtained from a connected bipartite graph with bipartition \(X\sqcup Y\) by attaching at least one leaf edge to each vertex of \(X\) and an arbitrary number of pendant triangles to each vertex of \(Y\). We show that for these graphs the value \(P_G(-1)\) is always \(\pm1\), and hence the pseudo-Gorenstein\(^{*}\) property is controlled by a simple parity condition.

We adopt the following notation throughout this section. 

\begin{notation}\label{not:CW}
Let $G$ be a Cameron--Walker graph with bipartite core $X\sqcup Y$ where
$X=\{x_1,\dots,x_n\}$ and $Y=\{y_1,\dots,y_m\}$.
Assume that each $x_i$ has $f_i\ge 1$ leaf neighbors attached, and each $y_j$ has $t_j\ge 0$
pendant triangles attached. For each \(i\), let \(L_i\) denote the set of the \(f_i\ge 1\) leaves attached to \(x_i\). 
Set
\[
F:=\sum_{i=1}^n f_i,
\qquad
T:=\sum_{j=1}^m t_j,
\qquad
m_0:=|\{\,j\in[m]:t_j=0\,\}|.
\]
\end{notation}

\begin{remark}\label{rem:CW_indep_number}
Let $G$ be a Cameron-Walker graph as in \Cref{not:CW}. The independence number of $G$ is
 \[
 \alpha(G)=F+T+m_0.
 \]
 For the lower bound, take all $F$ leaves attached to the $x_i$'s; for each $y_j$ with $t_j=0$ take $y_j$; for each $y_j$ with $t_j\ge 1$, take one non-$y_j$ vertex from each of its $t_j$ pendant triangles. This gives an independent set of size $F+T+m_0$.
 Now we focus on the upper bound. For each $i$, from the block $\{x_i\}\cup L_i$, any independent set uses at most $f_i$ vertices. For each $j$, from the $y_j$-block, any independent set uses at most 1 if $t_j=0$, and at most $t_j$ if $t_j\geq 1$. Thus, we have
 \[
 \alpha (G) \leq \sum_{i=1}^n f_i +\sum_{j} \max\{1,t_j\} = F+T+m_0.
 \]
\end{remark}

The following theorem recovers \cite[Theorem 1.1]{hibi2021CW} and gives a shorter proof using our independence polynomial approach.

\begin{thm}\label{thm:CW_a}
  Let $G$ be a Cameron-Walker graph. Then \(P_G(-1)=(-1)^{n+T}\in\{\pm 1\}\).  In particular, $\mathfrak a (G)=0$.
\end{thm}

\begin{proof}
Notice that every independent set $I$ of $G$ has a unique core intersection  $S=I\cap (X\sqcup Y)$. This $S$ is independent in the core.  For each independent set $S\subseteq X\sqcup Y$ in the bipartite core, let $C_S(x)$ denote the total contribution to $P_G(x)$ from all independent sets $I$ of $G$ satisfying $I\cap (X\sqcup Y)=S$. Then
\[
P_G(x)=\sum_{I} x^{|I|}= \sum_{S} C_S(x),
\]
where the first sum runs over all independent sets $I$ of $G$ and the second sum runs over all independent sets $S$ of the core.

We now compute $C_S(x)$. Fix such an $S$.  If $x_i\in S$, then no vertex of $L_i$ can lie in $I$. If $x_i\notin S$, then any subset of $L_i$ may be chosen, so the total contribution from the choices on $L_i$ is \((1+x)^{f_i}\). Similarly, if $y_j\in S$, then no vertex from any pendant triangle attached to $y_j$ can lie in $I$. If $y_j\notin S$, then in each pendant triangle attached to $y_j$ we may choose either no additional vertex or exactly one of the two vertices distinct from $y_j$. Hence each such triangle contributes a factor $1+2x$, and the $t_j$ pendant triangles attached to $y_j$ contribute \((1+2x)^{t_j}\). Therefore
\[
C_S(x)=x^{|S|}\cdot \prod_{x_i\notin S}(1+x)^{f_i}\cdot \prod_{y_j\notin S}(1+2x)^{t_j}.
\]
Now evaluate at $x=-1$. We obtain
\[
C_S(-1)=(-1)^{|S|}\cdot \prod_{x_i\notin S}(1-1)^{f_i}\cdot \prod_{y_j\notin S}(1-2)^{t_j}.
\]
Since each $f_i\ge 1$, if some $x_i\notin S$ then the factor $(1-1)^{f_i}=0$ appears. Hence $C_S(-1)=0$. Thus $C_S(-1)\neq 0$ can occur only if $x_i\in S$ for every $i$, that is, only if $X\subseteq S$. But $S$ is independent in the bipartite core. Since the core is connected, every $y_j\in Y$ has a neighbor in $X$, so no $y_j$ can belong to an independent set containing all of $X$. Hence the only independent set $S\subseteq X\sqcup Y$ with $C_S(-1)\neq 0$ is $S=X$.

It follows that
\[
P_G(-1)=C_X(-1)=(-1)^{|X|}\prod_{j=1}^m(1-2)^{t_j}
= (-1)^n\prod_{j=1}^m(-1)^{t_j}
= (-1)^{n+T}.
\]
In particular, $P_G(-1)\neq 0$, so $M(G)=0$ which is equivalent to $\mathfrak a (G)=0$.
\end{proof}

\begin{cor}\label{thm:CW_eval_minus_one}
A Cameron-Walker graph $G$  is pseudo-Gorenstein$^{*}$
if and only if $n+F+m_0 $ is even.
\end{cor}

\begin{proof}
By \Cref{cor:PGstar-criterion}, \(G\) is pseudo-Gorenstein\(^{*}\) if and only if \((-1)^{n+T}=(-1)^{\alpha(G)}\). Using \Cref{rem:CW_indep_number}, this is equivalent to
\[
(-1)^{n+T}=(-1)^{F+T+m_0},
\]
that is, to \(n+F+m_0\) being even.
\end{proof}

Having classified pseudo-Gorenstein$^{*}$ graphs in several natural families, we now study how the property behaves under graph operations, beginning with $C$-suspension.

\section{Suspension over vertex covers}

In this section, we study the changes to the  pseudo-Gorenstein$^{*}$ property under suspensions over vertex covers.  This graph operation and its impacts on regularity and projective dimension were studied in \cite{IndMatch} and \cite{Kara_Vien_2026}. 

We first recall the definition of suspension.

\begin{defn}
Let \(G=(V(G),E(G))\) be a finite simple graph and let \(\varnothing\neq C\subseteq V(G)\).
The \emph{\(C\)-suspension} of \(G\), denoted \(G(C)\), is the graph obtained from \(G\)
by adjoining a new vertex \(z\) adjacent exactly to the vertices of \(C\), i.e.,
\[
E(G(C))=E(G)\cup \{\{z,c\}: c\in C\}.
\]
\end{defn}

Write \(G\setminus C\) for the induced subgraph on \(V(G)\setminus C\).

\begin{remark}\label{rem:C-suspension}
Let \(z\) be the suspension vertex. Then \(G(C)-z=G\) and \(G(C)-N[z]=G\setminus C\).
Hence, by \Cref{lem:delete_contract},
\begin{equation}\label{eq:C-suspension-indpoly}
P_{G(C)}(x)=P_G(x)+x\,P_{G\setminus C}(x).
\end{equation}
Moreover,
\[
\alpha(G(C))=\max\{\alpha(G),\,1+\alpha(G\setminus C)\}.
\]
\end{remark}

As we see below, pseudo-Gorenstein$^{*}$ property behaves well under suspensions over vertex covers with the exceptions at the extremal vertex covers.

\begin{thm}\label{thm:co-suspension-preserves-PG-general}
Let \(G\) be a graph with \(\alpha(G)\ge 2\). Let $C$ be a vertex cover of $G$ such that \(
1\le|S|\le \alpha-1\) where $S= V(G) \setminus C$.
Let \(H:=G(C)\) be the \(C\)-suspension of \(G\).
Then
\[
G \text{ is pseudo-Gorenstein}^{*}\ \Longleftrightarrow\ H \text{ is pseudo-Gorenstein}^{*}.
\]
\end{thm}

\begin{proof}
Set $\alpha:=\alpha(G)$ and $s:=|S|$. We first show that \(\alpha(H)=\alpha(G)=\alpha\).
Let $z$ be the suspension vertex.  If \(A\) is an independent set of \(H\) containing \(z\), then
\(A\setminus\{z\}\subseteq S\), since \(z\) is adjacent to every vertex of \(C\).
Hence
\[
|A|\le1+s\le \alpha.
\]
If \(A\) is an independent set of \(H\) not containing \(z\), then \(A\) is an
independent set of \(G\), so \(|A|\le \alpha\).
Thus \(\alpha(H)\le \alpha\). Since every independent set of \(G\) is also
independent in \(H\), we have \(\alpha(H)\ge \alpha\). Therefore \(\alpha(H)=\alpha\).

Next, we compute the independence polynomial of \(H\).
Since \(C=V(G)\setminus S\), the induced subgraph \(G\setminus C=G|_S\) has no
edges. So
\[
P_{G\setminus C}(x)=P_{G|_S}(x)=(1+x)^s.
\]
Hence, by \Cref{eq:C-suspension-indpoly},
\[
P_H(x)=P_G(x)+xP_{G\setminus C}(x)=P_G(x)+x(1+x)^s.
\]
Evaluating at \(x=-1\), and using \(s\ge 1\), we get
\begin{equation}\label{eq:PG-general-at-minus-one}
P_H(-1)=P_G(-1).
\end{equation}

Now we compare the \(h\)-polynomials as follows:
\begin{align*}
h_{S/I(H)}(t)
&=(1-t)^{\alpha}P_H\!\left(\frac{t}{1-t}\right)\\
&=(1-t)^{\alpha}\left[
P_G\!\left(\frac{t}{1-t}\right)+\frac{t}{1-t}
\left(1+\frac{t}{1-t}\right)^s ~
\right] \\
&=(1-t)^{\alpha}P_G\!\left(\frac{t}{1-t}\right)
+(1-t)^{\alpha}\cdot\frac{t}{1-t}\cdot
\left(\frac{1}{1-t}\right)^s\\
&=h_{R/I(G)}(t)+t(1-t)^{\alpha-s-1}.
\end{align*}
Since \(s\le \alpha-1\), the exponent \(\alpha-s-1\) is nonnegative. In addition, since
\(s\ge 1\), the polynomial \(t(1-t)^{\alpha-s-1}\) has degree \(\alpha-s\le \alpha-1\).
Hence
\begin{equation}\label{eq:top-h-general}
h_{\alpha}(H)=h_{\alpha}(G).
\end{equation}

Finally, if \(G\) is pseudo-Gorenstein\(^{*}\), then \(\mathfrak a(G)=0\), so
\(P_G(-1)\neq 0\) by \Cref{thm:deg-via-ord}. By
\Cref{eq:PG-general-at-minus-one}, this implies \(P_H(-1)\neq 0\), hence
\(\mathfrak a(H)=0\). Also, by \Cref{eq:top-h-general}, we have
\[
h_{\alpha}(H)=h_{\alpha}(G)=1.
\]
Thus \(H\) is pseudo-Gorenstein\(^{*}\). The converse is identical, using again
\Cref{eq:PG-general-at-minus-one} and \Cref{eq:top-h-general}. This completes
the proof.
\end{proof}

\begin{remark}\label{rem:susp_ver_cover_extreme_cases}
    The two remaining cases from \Cref{thm:co-suspension-preserves-PG-general} are $|S|=0$ and $|S|=\alpha(G)$.

    If $|S|=0$, then $C=V(G)$ and $H$ is the full suspension of $G$. Then 
    $$P_H(x)= P_G(x)+x.$$
    Thus, $P_H(-1)= P_G(-1)-1$. Even when $\mathfrak a(G)=0$, which is equivalent to $P_G(-1)\neq 0$, it is possible to have $P_H(-1)=0$. If that happens, we have $\mathfrak a(H)<0$.  So,  pseudo-Gorenstein\(^{*}\) property is not necessarily preserved under full suspension.  

    When $|S|=\alpha(G)$, then $\alpha(H)=\alpha(G)+1$ and 
    $$h_{S/I(H)} (t)= (1-t) h_{R/I(G)} (t) + t.$$
    In this case, if $G$ is  pseudo-Gorenstein\(^{*}\), then $h_{\alpha} (G)=1$. Notice that the leading coefficient of $h_{S/I(H)} (t)$ is $h_{\alpha(G)+1} (H)= -1$. So, $H$ cannot be pseudo-Gorenstein\(^{*}\).
\end{remark}

As noted in \Cref{rem:susp_ver_cover_extreme_cases}, full suspension does not
preserve the pseudo-Gorenstein\(^{*}\) property in general. We therefore
determine exactly when the full suspensions of cycles and paths are
pseudo-Gorenstein\(^{*}\).

\begin{thm}\label{thm:PGstar-full-susp-path-cycle}
Let \(\widehat{G}\) denote the \emph{full suspension} (cone) over a graph \(G\).
\begin{enumerate}
\item[(a)] For \(n\ge 3\), the full suspension \(\widehat{C_n}\) is pseudo-Gorenstein$^{*}$
if and only if \(n\equiv 0 \pmod{12}\).

\item[(b)] For \(n\ge 1\), the full suspension \(\widehat{P_n}\) is pseudo-Gorenstein$^{*}$
if and only if \(n\equiv 1,10 \pmod{12}\).
\end{enumerate}
\end{thm}

\begin{proof}
Let \(z\) be the new vertex in the full suspension \(\widehat{G}\).
Since \(z\) is adjacent to every vertex of \(G\), we have \(P_{\widehat{G}}(x)=P_G(x)+x\). Hence
\begin{equation}\label{eq:cone-eval-minus-one}
P_{\widehat{G}}(-1)=P_G(-1)-1.
\end{equation}
Moreover, \(\alpha(\widehat{G})=\alpha(G)\), because any independent set of
\(\widehat{G}\) containing \(z\) has size \(1\), while every nonempty graph \(G\)
has an independent set of size at least \(1\). Next, we now apply \Cref{cor:PGstar-criterion}.

(a) Let \(G=C_n\), where \(n\ge 3\). Then
 \(\alpha(\widehat{C_n})=\alpha(C_n)=\left\lfloor \frac n2\right\rfloor\). By \Cref{cor:PGstar-criterion}, \(\widehat{C_n}\) is pseudo-Gorenstein\(^{*}\) if and only if
\(P_{\widehat{C_n}}(-1)=(-1)^{\lfloor n/2\rfloor}
\). Using \eqref{eq:cone-eval-minus-one}, this is equivalent to
\[
P_{C_n}(-1)-1=(-1)^{\lfloor n/2\rfloor}.
\]
By \Cref{lem:cycle-minus-one}, we have the following
\[
P_{C_n}(-1)-1=
\begin{cases}
1,& n\equiv 0 \pmod 6,\\[4pt]
0,& n\equiv 1,5 \pmod 6,\\[4pt]
-2,& n\equiv 2,4 \pmod 6,\\[4pt]
-3,& n\equiv 3 \pmod 6.
\end{cases}
\]
Since the right-hand side \( (-1)^{\lfloor n/2\rfloor}\) is always \(\pm 1\), the
only possible case is \(n\equiv 0\pmod 6\), where we must have
\[
1=(-1)^{\lfloor n/2\rfloor}.
\]
If \(n\equiv 0\pmod 6\), then \(\lfloor n/2\rfloor=n/2\), so this occurs exactly
when \(n/2\) is even, i.e. when \(n\equiv 0\pmod{12}\). Hence
\(\widehat{C_n}\) is pseudo-Gorenstein\(^{*}\) if and only if \(n\equiv 0\pmod{12}\).

\smallskip
\noindent
\textbf{(b)} Let \(G=P_n\), where \(n\ge 1\). Then
 \(\alpha(\widehat{P_n})=\alpha(P_n)=\left\lceil \frac n2\right\rceil\).
Again by \Cref{cor:PGstar-criterion}, \(\widehat{P_n}\) is pseudo-Gorenstein\(^{*}\)
if and only if \(P_{\widehat{P_n}}(-1)=(-1)^{\lceil n/2\rceil}\). Using \eqref{eq:cone-eval-minus-one}, this is equivalent to
\[
P_{P_n}(-1)-1=(-1)^{\lceil n/2\rceil}.
\]
By \Cref{lem:path-minus-one}, we have
\[
P_{P_n}(-1)-1=
\begin{cases}
0,& n\equiv 0,5 \pmod 6,\\[4pt]
-1,& n\equiv 1,4 \pmod 6,\\[4pt]
-2,& n\equiv 2,3 \pmod 6.
\end{cases}
\]
Thus equality with \( (-1)^{\lceil n/2\rceil}\in\{\pm1\}\) is possible only when
\(n\equiv 1,4\pmod 6\), in which case we need
\[
-1=(-1)^{\lceil n/2\rceil}.
\]
If \(n\equiv 1\pmod 6\), then this happens exactly for \(n\equiv 1\pmod{12}\).
If \(n\equiv 4\pmod 6\), then this happens exactly for \(n\equiv 10\pmod{12}\).
Therefore \(\widehat{P_n}\) is pseudo-Gorenstein\(^{*}\) if and only if \(n\equiv 1,10 \pmod{12}\).
\end{proof}

\section{Suspensions over maximal independent sets: cycles}

In this section, we study suspensions over maximal independent sets of cycles.
In general, suspensions over maximal independent sets do not preserve the
\(\mathfrak a\)-invariant (see \cite[Theorem 5.7]{Kara_Vien_2026}).

Let \(n\ge 4\), let \(\mathcal C\) be a maximal independent set of \(C_n\), and
let \(G\) be the \(\mathcal C\)-suspension of \(C_n\) with new vertex \(z\), i.e.
\[
V(G)=V(C_n)\cup\{z\},
\qquad
E(G)=E(C_n)\cup\{\{z,x_i\}:x_i\in \mathcal C\}.
\]
Recall from \Cref{cor:path-cycle-signed-values} that, with \(\alpha:=\alpha(C_n)=\left\lfloor \frac n2\right\rfloor\)
\[
a_n:=(-1)^{\alpha}P_{C_n}(-1)=h_{\alpha}(C_n),
\]
the value of \(a_n\) is determined by \(n \bmod 12\).

\begin{remark}\label{rem:cycle-max-ind}
Write \(c:=|\mathcal C|\). Then \(\left\lceil \frac n3\right\rceil \le c\le \left\lfloor \frac n2\right\rfloor\). Set \(\ell:=n-2c\). Since \(\mathcal C\) is maximal independent, the gaps (vertices) between
consecutive vertices of \(\mathcal C\) around the cycle have length \(1\) or \(2\).
Hence, if \(H:=G-N[z]\), then \(H\) is the disjoint union of \(\ell\) edges and \(c-\ell\) isolated vertices.
Thus
\[
\alpha(H)=\ell+(c-\ell)=c. 
\]
Therefore \(\alpha(G)=\max\{\alpha(C_n),\,1+\alpha(H)\}=\max\{\alpha,\,c+1\}\). Since \(c\le \alpha\), it follows that
\[
\alpha(G)=
\begin{cases}
\alpha+1,& \text{if } c=\alpha,\\[4pt]
\alpha,& \text{if } c<\alpha.
\end{cases}
\]
\end{remark}

\begin{remark}
The case \(n=3\) is exceptional. Here \(|\mathcal C|=1\), and the unique
\(\mathcal C\)-suspension \(G\) of \(C_3\) satisfies
\[
P_G(x)=1+4x+2x^2,
\qquad \text{ and } \qquad
h_G(t)=1+2t-t^2.
\]
In particular, \(G\) is not pseudo-Gorenstein\(^{*}\).
\end{remark}

\begin{thm}\label{thm:suspension_cycle}
We have
\[
h_{\alpha(G)}(G)=
\begin{cases}
-a_n,& \text{if } c=\alpha,\\[6pt]
\operatorname{sgn}(a_n),& \text{if } c=\left\lceil \dfrac n3\right\rceil<\alpha,\\[6pt]
a_n,& \text{if } \left\lceil \dfrac n3\right\rceil<c<\alpha.
\end{cases}
\]
\end{thm}

\begin{proof}
By \Cref{lem:delete_contract}, since \(G-z=C_n\) and \(G-N[z]=H\), we have
\[
P_G(x)=P_{C_n}(x)+x\,P_H(x)
      =P_{C_n}(x)+x(1+x)^{c-\ell}(1+2x)^{\ell}.
\]
Hence, by \Cref{prop:h-from-independence},
\[
h_{\alpha(G)}(G)=(-1)^{\alpha(G)}P_G(-1).
\]

If \(c>\ell\), then \(3c>n\). So \(c-\ell>0\), and the second term vanishes at
\(x=-1\). Thus
\[
P_G(-1)=P_{C_n}(-1),
\qquad
h_{\alpha(G)}(G)=(-1)^{\alpha(G)-\alpha}a_n.
\]
If \(c=\alpha\), then \(\alpha(G)=\alpha+1\). So \(h_{\alpha(G)}(G)=-a_n\).
If \(\left\lceil \frac n3\right\rceil<c<\alpha\), then \(\alpha(G)=\alpha\). So
\(h_{\alpha(G)}(G)=a_n\).
Finally, if \(c=\left\lceil \frac n3\right\rceil<\alpha\), then \(3\nmid n\). So
\(a_n\in\{\pm1\}\) by \Cref{cor:path-cycle-signed-values}. Hence in this subcase
\(h_{\alpha(G)}(G)=a_n=\operatorname{sgn}(a_n)\).

It remains to consider \(c=\ell\). In this case \(n=3c\). So \(c=\left\lceil \frac n3\right\rceil<\alpha\)
and \(\alpha(G)=\alpha\). Evaluating at \(x=-1\), we get
\[
P_G(-1)=P_{C_n}(-1)+(-1)^{c+1} \qquad \text{ and } \qquad
h_{\alpha(G)}(G)=a_n+(-1)^{\alpha+c+1}.
\]
Since \(n=3c\), the values of \(P_{C_n}(-1)\) from \Cref{lem:cycle-minus-one} give \(P_{C_n}(-1)=2(-1)^c\). Therefore
\[
a_n=(-1)^{\alpha}P_{C_n}(-1)=2(-1)^{\alpha+c}.
\]
Thus
\[
h_{\alpha(G)}(G)
=
2(-1)^{\alpha+c}+(-1)^{\alpha+c+1}
=
(-1)^{\alpha+c}
=
\operatorname{sgn}(a_n).
\]
This proves the theorem.
\end{proof}

\begin{cor}
Let \(n\ge 4\), let \(\mathcal C\) be a maximal independent set of \(C_n\), and
let \(G\) be the \(\mathcal C\)-suspension of \(C_n\). Then \(G\) is
pseudo-Gorenstein\(^{*}\) if and only if one of the following holds:
\begin{enumerate}
\item[(a)] \(n\equiv 0,3 \pmod{12}\) and \(|\mathcal C|=\left\lceil \frac n3\right\rceil\);
\item[(b)] \(n\equiv 4,7,8,11 \pmod{12}\) and \(|\mathcal C|=\alpha(C_n)\);
\item[(c)] \(n\equiv 1,2,5,10 \pmod{12}\) and \(|\mathcal C|\ne \alpha(C_n)\).
\end{enumerate}
\end{cor}

\begin{proof}
It was shown in \cite[Theorem 4.9]{Kara_Vien_2026} that \(\mathfrak a(G)=0\) for any $\mathcal{C}$-suspension of $C_n$ for any $n\geq 3$. Hence \(G\) is pseudo-Gorenstein\(^{*}\) if and
only if \(h_{\alpha(G)}(G)=1\).

Now apply \Cref{thm:suspension_cycle} together with the values of \(a_n\) from
\Cref{cor:path-cycle-signed-values}. If \(c=\alpha\), then
\(h_{\alpha(G)}(G)=-a_n\), which equals \(1\) exactly when
\(n\equiv 4,7,8,11\pmod{12}\). If \(c=\left\lceil \frac n3\right\rceil<\alpha\),
then \(h_{\alpha(G)}(G)=\operatorname{sgn}(a_n)\), which equals \(1\) exactly when
\(n\equiv 0,1,2,3,5,10\pmod{12}\). If
\(\left\lceil \frac n3\right\rceil<c<\alpha\), then \(h_{\alpha(G)}(G)=a_n\), which
equals \(1\) exactly when \(n\equiv 1,2,5,10\pmod{12}\). Combining these cases gives
(a), (b), and (c).
\end{proof}

\section{Suspensions over maximal independent sets: paths}

In this section, we study suspensions over maximal independent sets of paths.

Let \(n\ge 2\), let \(P_n\) be the path on vertices \(\{x_1,\ldots,x_n\}\), let
\(\mathcal C\) be a maximal independent set of \(P_n\), and let \(G\) be the
\(\mathcal C\)-suspension of \(P_n\) with new vertex \(z\).

Set \(\alpha:=\alpha(P_n)=\left\lceil \frac n2\right\rceil\). Recall from  \Cref{cor:path-cycle-signed-values}  that
\[
b_n:=(-1)^{\alpha}P_{P_n}(-1)=h_{\alpha}(P_n),
\]
the values of \(b_n\) depend only on
\(n\bmod 12\).

\begin{notation}\label{not:path_susp_max_ind}
Write \(\mathcal C=\{x_{i_1},x_{i_2},\dots,x_{i_c}\}\) with
\(i_1<i_2<\cdots<i_c\), where \(c:=|\mathcal C|\). For consecutive vertices
\(x_{i_j},x_{i_{j+1}}\in \mathcal C\), call the collection of vertices \(\{x_{i_j+1},\dots,x_{i_{j+1}-1}\}\) the \emph{internal gap} between them. Since \(\mathcal C\) is maximal
independent, each such gap has size \(1\) or \(2\), equivalently
\[
2\le i_{j+1}-i_j\le 3 \qquad \text{for }1\le j\le c-1.
\]
Moreover, the left endpoint \(x_1\) is either in \(\mathcal C\) or is the unique
vertex before \(x_{i_1}\). An analogous statement holds for the right endpoint \(x_n\). Define
\[
\delta_0=
\begin{cases}
0,& \text{if }x_1\in \mathcal C,\\
1,& \text{if }x_1\notin \mathcal C,
\end{cases}
\qquad
\delta_t=
\begin{cases}
0,& \text{if }x_n\in \mathcal C,\\
1,& \text{if }x_n\notin \mathcal C,
\end{cases}
\]
and set \(\delta:=\delta_0+\delta_t\in\{0,1,2\}\).
\end{notation}

\begin{remark}\label{rem:delta}
Let \(H:=G-N[z]\). Then \(H\) is obtained from \(P_n\) by deleting the vertices of
\(\mathcal C\). Every internal gap of size \(2\) contributes an edge to \(H\),
while every internal gap of size \(1\) contributes an isolated vertex. In
addition, each missing endpoint contributes one more isolated vertex. If \(\ell\)
denotes the number of internal gaps of size \(2\), then \(H\) is the disjoint
union of \(\ell\) edges and
\[
e:=(c-1-\ell)+\delta=c-\ell-1+\delta
\]
isolated vertices. Consequently, \(\alpha(H)=\ell+e=c-1+\delta\) and
\[
P_H(x)=(1+x)^e(1+2x)^\ell.
\]
Hence \(\alpha(G)=\max\{\alpha(P_n),\,1+\alpha(H)\}
=\max\{\alpha,\,c+\delta\}\).
Moreover,
\[
\alpha(G)=
\begin{cases}
\alpha+1,& \text{if }c+\delta=\alpha+1,\\[4pt]
\alpha,& \text{if }c+\delta\le \alpha.
\end{cases}
\]
Indeed,   \(n=2c+\ell-1+\delta\ge 2c-1+\delta\). Hence \(2c+\delta\le n+1\).
Since \(\delta\le 2\), we get
\[
2(c+\delta)=2c+2\delta\le n+1+\delta\le n+3.
\]
Therefore
\[
c+\delta\le \left\lfloor \frac{n+3}{2}\right\rfloor
=\left\lceil \frac n2\right\rceil+1=\alpha+1.
\]
\end{remark}

\begin{thm}\label{thm:suspension_path}
Let \(G\) be a \(\mathcal C\)-suspension of \(P_n\) with notation as in
\Cref{not:path_susp_max_ind}.
\begin{enumerate}[label=(\alph*)]
\item If \(n\equiv 0,2 \pmod 3\), then \(\mathfrak a(G)=0\) and
\[
h_{\alpha(G)}(G)=
\begin{cases}
-b_n,& \text{if }c+\delta=\alpha+1,\\[4pt]
b_n,& \text{if }c+\delta\le \alpha.
\end{cases}
\]

\item If \(n=3k+1\) and \(\mathcal C=\{x_1,x_4,\dots,x_{3k+1}\}\), then
\(\mathfrak a(G)=0\) and
\[
h_{\alpha(G)}(G)=(-1)^{\alpha+k+1}.
\]

\item In all remaining cases, \(\mathfrak a(G)<0\).
\end{enumerate}
\end{thm}

\begin{proof}
It is proved in \cite[Theorems 5.6 and 5.7]{Kara_Vien_2026} that \(\mathfrak a(G)=0\)  exactly when \(n\equiv 0,2\pmod 3\) or when  \(n=3k+1\) and \(\mathcal C=\{x_1,x_4,\dots,x_{3k+1}\}\)
for some \(k\ge 1\).  This proves part (c). Thus, it remains to consider for which $n$ the following holds: 
\[
1= h_{\alpha(G)}(G)=(-1)^{\alpha(G)}P_G(-1).
\]
It follows from \Cref{lem:delete_contract} that
\[
P_G(x)=P_{P_n}(x)+x\,P_H(x)
      =P_{P_n}(x)+x(1+x)^e(1+2x)^\ell.
\]
First suppose \(e\ge 1\). Then  \(P_G(-1)=P_{P_n}(-1)\). By \Cref{cor:path-cycle-signed-values}, this is nonzero exactly when
\(n\equiv 0,2\pmod 3\). By \Cref{prop:h-from-independence} 
\[
h_{\alpha(G)}(G)=(-1)^{\alpha(G)}P_G(-1)
= (-1)^{\alpha(G)-\alpha}b_n.
\]
Now apply \Cref{rem:delta}: if \(c+\delta=\alpha+1\), then
\(\alpha(G)=\alpha+1\), so \(h_{\alpha(G)}(G)=-b_n\); if \(c+\delta\le \alpha\),
then \(\alpha(G)=\alpha\), so \(h_{\alpha(G)}(G)=b_n\). This proves (a).

Now suppose \(e=0\). Then \(H\) has no isolated vertices. This means \(\delta=0\) and
every internal gap has size \(2\) which is equivalent to \(\mathcal C=\{x_1,x_4,\dots,x_{3k+1}\}\)
for some \(k\ge 1\). Therefore \(n=3k+1\). In this case \(H\) is the
disjoint union of \(k\) edges with \(\ell=k\). Since \(n\equiv 1\pmod 3\), we have
\(P_{P_n}(-1)=0\) and 
\[
P_G(-1)=(-1)(1-2)^k=(-1)^{k+1}\neq 0.
\]
Since  \(c=k+1\le \alpha\), we have
\(\alpha(G)=\alpha\) by \Cref{rem:delta}. Hence, (b) holds:
\[
h_{\alpha(G)}(G)=(-1)^{\alpha(G)}P_G(-1)=(-1)^{\alpha+k+1}. \qedhere
\]
\end{proof}

\begin{cor}\label{cor:suspension_path}
Suppose \(n\ge 2\). Then \(G\) is pseudo-Gorenstein\(^{*}\) if and only if one of
the following holds:
\begin{enumerate}[label=(\alph*)]
\item \(n\equiv 0,2,9,11 \pmod{12}\) and \(c+\delta\le \alpha\);

\item \(n\equiv 3,5,6,8 \pmod{12}\) and \(c+\delta=\alpha+1\);

\item \(n\equiv 1,4 \pmod{12}\) and \(\mathcal C=\{x_1,x_4,\dots,x_n\}\).
\end{enumerate}
\end{cor}

\begin{proof}
By \Cref{cor:path-cycle-signed-values},
\[
b_n=
\begin{cases}
1,& \text{if } n\equiv 0,2,9,11 \pmod{12},\\[4pt]
-1,& \text{if } n\equiv 3,5,6,8 \pmod{12},\\[4pt]
0,& \text{if } n\equiv 1,4,7,10 \pmod{12}.
\end{cases}
\]

If \(n\equiv 0,2\pmod 3\), then \Cref{thm:suspension_path}(a) shows that
\(G\) is pseudo-Gorenstein\(^{*}\) exactly when \(h_{\alpha(G)}(G)=1\). If
\(c+\delta\le \alpha\), this means \(b_n=1\), giving
\(n\equiv 0,2,9,11\pmod{12}\). If \(c+\delta=\alpha+1\), this means
\(-b_n=1\), giving \(n\equiv 3,5,6,8\pmod{12}\).

If \(n=3k+1\) and \(\mathcal C=\{x_1,x_4,\dots,x_{3k+1}\}\), then by
\Cref{thm:suspension_path}(b), \(G\) is pseudo-Gorenstein\(^{*}\) exactly when
\[
(-1)^{\alpha+k+1}=1.
\]
A parity check shows that this holds exactly when \(k\equiv 0,1\pmod 4\),
equivalently when \(n=3k+1\equiv 1,4\pmod{12}\).

All remaining cases are excluded by \Cref{thm:suspension_path}(c).
\end{proof}

\bibliographystyle{abbrv}
\bibliography{ref}

@misc{Kara_Vien_2026,
      title={Algebraic Invariants of Edge Ideals Under Suspension}, 
      author={Selvi Kara and Dalena Vien},
      year={2026},
      eprint={2603.05657},
      archivePrefix={arXiv},
      primaryClass={math.AC},
      url={https://arxiv.org/abs/2603.05657}, 
}

@article{GutmanHarary,
    AUTHOR = {Gutman, Ivan and Harary, Frank},
     TITLE = {Generalizations of the matching polynomial},
   JOURNAL = {Utilitas Math.},
    VOLUME = {24},
      YEAR = {1983},
     PAGES = {97--106},
}

@article {Unicyclic,
    AUTHOR = {Pedersen, Anders Sune and Vestergaard, Preben Dahl},
     TITLE = {The number of independent sets in unicyclic graphs},
   JOURNAL = {Discrete Appl. Math.},
  FJOURNAL = {Discrete Applied Mathematics. The Journal of Combinatorial
              Algorithms, Informatics and Computational Sciences},
    VOLUME = {152},
      YEAR = {2005},
    NUMBER = {1-3},
     PAGES = {246--256},
      ISSN = {0166-218X,1872-6771},
   MRCLASS = {05C69 (05C05 11B39)},
  MRNUMBER = {2174205},
MRREVIEWER = {Reinhardt\ Euler},
       DOI = {10.1016/j.dam.2005.04.002},
       URL = {https://doi.org/10.1016/j.dam.2005.04.002},
}

@article {Application_music_ind_poly,
    AUTHOR = {Brown, Jason and Hoshino, Richard},
     TITLE = {Independence polynomials of circulants with an application to
              music},
   JOURNAL = {Discrete Math.},
  FJOURNAL = {Discrete Mathematics},
    VOLUME = {309},
      YEAR = {2009},
    NUMBER = {8},
     PAGES = {2292--2304},
      ISSN = {0012-365X,1872-681X},
   MRCLASS = {05C69 (05A15 05C30 05C31)},
  MRNUMBER = {2510357},
MRREVIEWER = {Eugen\ Mandrescu},
       DOI = {10.1016/j.disc.2008.05.003},
       URL = {https://doi.org/10.1016/j.disc.2008.05.003},
}

@article {Unimodal_ind_poly,
    AUTHOR = {Brown, J. I. and Cameron, B.},
     TITLE = {On the unimodality of independence polynomials of very
              well-covered graphs},
   JOURNAL = {Discrete Math.},
  FJOURNAL = {Discrete Mathematics},
    VOLUME = {341},
      YEAR = {2018},
    NUMBER = {4},
     PAGES = {1138--1143},
      ISSN = {0012-365X,1872-681X},
   MRCLASS = {05C31 (05C69)},
  MRNUMBER = {3764366},
MRREVIEWER = {Saeid\ Alikhani},
       DOI = {10.1016/j.disc.2017.10.007},
       URL = {https://doi.org/10.1016/j.disc.2017.10.007},
}

@article {Claw_free_ind_poly,
    AUTHOR = {Chudnovsky, Maria and Seymour, Paul},
     TITLE = {The roots of the independence polynomial of a clawfree graph},
   JOURNAL = {J. Combin. Theory Ser. B},
  FJOURNAL = {Journal of Combinatorial Theory. Series B},
    VOLUME = {97},
      YEAR = {2007},
    NUMBER = {3},
     PAGES = {350--357},
      ISSN = {0095-8956,1096-0902},
   MRCLASS = {05C69},
  MRNUMBER = {2305888},
MRREVIEWER = {Steven\ D.\ Noble},
       DOI = {10.1016/j.jctb.2006.06.001},
       URL = {https://doi.org/10.1016/j.jctb.2006.06.001},
}

@article {Ind_complex,
    AUTHOR = {Bousquet-M\'elou, Mireille and Linusson, Svante and Nevo,
              Eran},
     TITLE = {On the independence complex of square grids},
   JOURNAL = {J. Algebraic Combin.},
  FJOURNAL = {Journal of Algebraic Combinatorics. An International Journal},
    VOLUME = {27},
      YEAR = {2008},
    NUMBER = {4},
     PAGES = {423--450},
      ISSN = {0925-9899,1572-9192},
   MRCLASS = {05C69 (05B45 05C50)},
  MRNUMBER = {2393250},
       DOI = {10.1007/s10801-007-0096-x},
       URL = {https://doi.org/10.1007/s10801-007-0096-x},
}

@article {Fibonaci_main,
    AUTHOR = {Prodinger, Helmut and Tichy, Robert F.},
     TITLE = {Fibonacci numbers of graphs},
   JOURNAL = {Fibonacci Quart.},
  FJOURNAL = {The Fibonacci Quarterly. Official Organ of the Fibonacci
              Association},
    VOLUME = {20},
      YEAR = {1982},
    NUMBER = {1},
     PAGES = {16--21},
      ISSN = {0015-0517},
   MRCLASS = {05C99 (05C05)},
  MRNUMBER = {660753},
MRREVIEWER = {\c Serban\ Buze\c teanu},
}

@article {Fibonacci,
    AUTHOR = {Knopfmacher, Arnold and Tichy, Robert F. and Wagner, Stephan
              and Ziegler, Volker},
     TITLE = {Graphs, partitions and {F}ibonacci numbers},
   JOURNAL = {Discrete Appl. Math.},
  FJOURNAL = {Discrete Applied Mathematics. The Journal of Combinatorial
              Algorithms, Informatics and Computational Sciences},
    VOLUME = {155},
      YEAR = {2007},
    NUMBER = {10},
     PAGES = {1175--1187},
      ISSN = {0166-218X,1872-6771},
   MRCLASS = {05C70 (05C69)},
  MRNUMBER = {2332311},
       DOI = {10.1016/j.dam.2006.10.010},
       URL = {https://doi.org/10.1016/j.dam.2006.10.010},
}

@article {Gen_choromatic_poly,
    AUTHOR = {Dohmen, Klaus and P\"onitz, Andr\'e{} and Tittmann, Peter},
     TITLE = {A new two-variable generalization of the chromatic polynomial},
   JOURNAL = {Discrete Math. Theor. Comput. Sci.},
  FJOURNAL = {Discrete Mathematics \& Theoretical Computer Science. DMTCS.},
    VOLUME = {6},
      YEAR = {2003},
    NUMBER = {1},
     PAGES = {69--89},
      ISSN = {1365-8050},
   MRCLASS = {05C15 (68Q25 68R10)},
  MRNUMBER = {1996108},
MRREVIEWER = {N.\ Z.\ Li},
}

@article {GW,
    AUTHOR = {Goto, Shiro and Watanabe, Keiichi},
     TITLE = {On graded rings. {I}},
   JOURNAL = {J. Math. Soc. Japan},
  FJOURNAL = {Journal of the Mathematical Society of Japan},
    VOLUME = {30},
      YEAR = {1978},
    NUMBER = {2},
     PAGES = {179--213},
      ISSN = {0025-5645,1881-1167},
   MRCLASS = {13H10 (13D03 14B15)},
  MRNUMBER = {494707},
MRREVIEWER = {Gerald\ S.\ Garfinkel},
       DOI = {10.2969/jmsj/03020179},
       URL = {https://doi.org/10.2969/jmsj/03020179},
}

@article {Sta78,
    AUTHOR = {Stanley, Richard P.},
     TITLE = {Hilbert functions of graded algebras},
   JOURNAL = {Advances in Math.},
  FJOURNAL = {Advances in Mathematics},
    VOLUME = {28},
      YEAR = {1978},
    NUMBER = {1},
     PAGES = {57--83},
      ISSN = {0001-8708},
   MRCLASS = {13D10 (13H10)},
  MRNUMBER = {485835},
MRREVIEWER = {Idun\ Reiten},
       DOI = {10.1016/0001-8708(78)90045-2},
       URL = {https://doi.org/10.1016/0001-8708(78)90045-2},
}

@book {Sta72,
    AUTHOR = {Stanley, Richard P.},
     TITLE = {Ordered structures and partitions},
    SERIES = {Memoirs of the American Mathematical Society},
    VOLUME = {No. 119},
 PUBLISHER = {American Mathematical Society, Providence, RI},
      YEAR = {1972},
     PAGES = {iii+104},
   MRCLASS = {05A17},
  MRNUMBER = {332509},
MRREVIEWER = {L.\ K.\ Durst},
}

@article {hibi2021CW,
    AUTHOR = {Hibi, Takayuki and Kimura, Kyouko and Matsuda, Kazunori and
              Tsuchiya, Akiyoshi},
     TITLE = {Regularity and {$a$}-invariant of {C}ameron-{W}alker graphs},
   JOURNAL = {J. Algebra},
  FJOURNAL = {Journal of Algebra},
    VOLUME = {584},
      YEAR = {2021},
     PAGES = {215--242},
      ISSN = {0021-8693,1090-266X},
   MRCLASS = {13A70 (05E40 13H10)},
  MRNUMBER = {4270555},
MRREVIEWER = {Aming\ Liu},
       DOI = {10.1016/j.jalgebra.2021.05.007},
       URL = {https://doi.org/10.1016/j.jalgebra.2021.05.007},
}

@misc{biermann2026,
      title={Realizable (reg, deg h)-Pairs for Cover Ideals via Independence Polynomials}, 
      author={Jennifer Biermann and Trung Chau and Selvi Kara and Augustine O'Keefe and Joseph Skelton and Gabriel Sosa Castillo and Dalena Vien},
      year={arXiv Preprint 2602.10376, 2026},
      eprint={2602.10376},
      archivePrefix={arXiv},
      primaryClass={math.AC},
      url={https://arxiv.org/abs/2602.10376}, 
}

@book {EC1,
    AUTHOR = {Stanley, Richard P.},
     TITLE = {Enumerative combinatorics. {V}olume 1},
    SERIES = {Cambridge Studies in Advanced Mathematics},
    VOLUME = {49},
   EDITION = {Second},
 PUBLISHER = {Cambridge University Press, Cambridge},
      YEAR = {2012},
     PAGES = {xiv+626},
      ISBN = {978-1-107-60262-5},
   MRCLASS = {05-02 (05A15 06-02)},
  MRNUMBER = {2868112},
}

@incollection {Hibi,
    AUTHOR = {Hibi, Takayuki},
     TITLE = {Distributive lattices, affine semigroup rings and algebras
              with straightening laws},
 BOOKTITLE = {Commutative algebra and combinatorics ({K}yoto, 1985)},
    SERIES = {Adv. Stud. Pure Math.},
    VOLUME = {11},
     PAGES = {93--109},
 PUBLISHER = {North-Holland, Amsterdam},
      YEAR = {1987},
      ISBN = {0-444-70314-4},
   MRCLASS = {13H10 (06D99)},
  MRNUMBER = {951198},
MRREVIEWER = {Keiichi\ Watanabe},
       DOI = {10.2969/aspm/01110093},
       URL = {https://doi.org/10.2969/aspm/01110093},
}

@article {EHHM,
    AUTHOR = {Ene, Viviana and Herzog, J\"urgen and Hibi, Takayuki and
              Saeedi Madani, Sara},
     TITLE = {Pseudo-{G}orenstein and level {H}ibi rings},
   JOURNAL = {J. Algebra},
  FJOURNAL = {Journal of Algebra},
    VOLUME = {431},
      YEAR = {2015},
     PAGES = {138--161},
      ISSN = {0021-8693,1090-266X},
   MRCLASS = {13H10 (05E40 06A11 06D99)},
  MRNUMBER = {3327545},
MRREVIEWER = {Daniel\ D.\ Anderson},
       DOI = {10.1016/j.jalgebra.2015.02.002},
       URL = {https://doi.org/10.1016/j.jalgebra.2015.02.002},
}

@article{IndMatch,
    AUTHOR = {Hibi, Takayuki and Kanno, Hiroju and Matsuda, Kazunori},
     TITLE = {Induced matching numbers of finite graphs and edge ideals},
   JOURNAL = {J. Algebra},
  FJOURNAL = {Journal of Algebra},
    VOLUME = {532},
      YEAR = {2019},
     PAGES = {311--322},
      ISSN = {0021-8693,1090-266X},
   MRCLASS = {05C70 (05C69 05E40 13D40 13H10)},
  MRNUMBER = {3959482},
MRREVIEWER = {Mehrdad\ Nasernejad},
       DOI = {10.1016/j.jalgebra.2019.04.036},
       URL = {https://doi.org/10.1016/j.jalgebra.2019.04.036},
}

@article {CW+Graphs,
    AUTHOR = {Cameron, Kathie and Walker, Tracy},
     TITLE = {The graphs with maximum induced matching and maximum matching
              the same size},
   JOURNAL = {Discrete Math.},
  FJOURNAL = {Discrete Mathematics},
    VOLUME = {299},
      YEAR = {2005},
    NUMBER = {1-3},
     PAGES = {49--55},
      ISSN = {0012-365X,1872-681X},
   MRCLASS = {05C70 (05C17 05C75)},
  MRNUMBER = {2168694},
MRREVIEWER = {Winfried\ Hochst\"attler},
       DOI = {10.1016/j.disc.2004.07.022},
       URL = {https://doi.org/10.1016/j.disc.2004.07.022},
}

@article{levit2009independence_preprint,
  title={The independence polynomial of a graph at-1},
  author={Levit, Vadim E and Mandrescu, Eugen},
  journal={arXiv preprint arXiv:0904.4819},
  year={2009}
}

@article {Levit_Mandrescu,
    AUTHOR = {Levit, Vadim E. and Mandrescu, Eugen},
     TITLE = {The cyclomatic number of a graph and its independence
              polynomial at {$-1$}},
   JOURNAL = {Graphs Combin.},
  FJOURNAL = {Graphs and Combinatorics},
    VOLUME = {29},
      YEAR = {2013},
    NUMBER = {2},
     PAGES = {259--273},
      ISSN = {0911-0119,1435-5914},
   MRCLASS = {05C31 (05C38 05C69)},
  MRNUMBER = {3027601},
MRREVIEWER = {Mohammad\ Reza\ Oboudi},
       DOI = {10.1007/s00373-011-1101-7},
       URL = {https://doi.org/10.1007/s00373-011-1101-7},
}

\end{document}